\documentclass[10pt,a4paper,reqno]{amsart}
\usepackage{graphicx}
\usepackage[linktocpage=true,colorlinks,citecolor=blue,linkcolor=blue,urlcolor=black,pagebackref]{hyperref}
\usepackage[margin=2cm,includehead]{geometry}
\usepackage{amssymb}
\usepackage{cite}
\usepackage{amsmath}
\usepackage{latexsym}
\usepackage{amscd}
\usepackage{amsthm}
\usepackage{mathrsfs}
\usepackage{url}
\usepackage{amsmath,mathtools}
\usepackage{cleveref}
\usepackage{hyperref}
\usepackage[alphabetic,backrefs]{amsrefs}
\usepackage[T1]{fontenc}

\newtheorem{thm}{Theorem}[section]
\newtheorem{corr}[thm]{Corollary}
\newtheorem{lemma}[thm]{Lemma}
\newtheorem{prop}[thm]{Proposition}

\theoremstyle{definition}
\newtheorem{defn}[thm]{Definition}
\newtheorem{ex}[thm]{Example}
\newtheorem{rem}[thm]{Remark}
\newtheorem*{ack}{Acknowledgments}
\numberwithin{equation}{section}
\newcommand{\RNum}[1]{\uppercase\expandafter{\romannumeral #1\relax}}

\def\pt{\partial}
\def\del{\nabla}
\def\G2{\GG_2}
\def\g2{\varphi}

\def\lieX{\mathcal{L}_X}

\DeclareMathOperator\GG{G}

\DeclareMathOperator\Div{div}

\DeclareMathOperator\Vol{Vol}

\DeclareMathOperator\tRic{\text{Ric}}

\crefformat{section}{\S#2#1#3} 
\crefformat{subsection}{\S#2#1#3}
\crefformat{subsubsection}{\S#2#1#3}

\makeatletter
\def\@tocline#1#2#3#4#5#6#7{\relax
  \ifnum #1>\c@tocdepth 
  \else
    \par \addpenalty\@secpenalty\addvspace{#2}%
    \begingroup \hyphenpenalty\@M
    \@ifempty{#4}{%
      \@tempdima\csname r@tocindent\number#1\endcsname\relax
    }{%
      \@tempdima#4\relax
    }%
    \parindent\z@ \leftskip#3\relax \advance\leftskip\@tempdima\relax
    \rightskip\@pnumwidth plus4em \parfillskip-\@pnumwidth
    #5\leavevmode\hskip-\@tempdima
      \ifcase #1
       \or\or \hskip 1.8em \or \hskip 4.7em \else \hskip 3em \fi%
      #6\nobreak\relax
    \hfill\hbox to\@pnumwidth{\@tocpagenum{#7}}\par 
    \nobreak
    \endgroup
  \fi}
\makeatother

\begin{document}
\title[Some results on RB solitons and almost solitons]{Some results on Ricci-Bourguignon solitons and almost solitons}
\author{Shubham Dwivedi}
\address{Department of Pure Mathematics, University of Waterloo, Waterloo, ON N2L3G1}
\email{\href{mailto:s2dwived@uwaterloo.ca}{s2dwived@uwaterloo.ca}}
\date{\today}

%
\begin{abstract}
We prove some results for the solitons of the Ricci-Bourguignon flow, generalizing the corresponding results for Ricci solitons. Taking motivation from Ricci almost solitons, we then introduce the notion of Ricci-Bourguignon \emph{almost} solitons and prove some results about them which generalize previous results for Ricci almost solitons. We also derive integral formulas for compact gradient Ricci-Bourguignon solitons and compact gradient Ricci-Bourguignon almost solitons. Finally, using the integral formula,  we show that a compact  gradient Ricci-Bourguignon almost soliton is isometric to a Euclidean sphere if it has constant scalar curvature or its associated vector field is conformal. 
\end{abstract}

\maketitle


\medskip


\section{Introduction}\label{sec-intro}

Ricci solitons play a major role in Ricci flow where they correspond to self-similar solutions of the flow. Thus, given a geometric flow, it is natural to study the solitons associated to that flow. A family of metrics $g(t)$ on an $n$-dimensional Riemannian manifold $(M^n,g)$ is said to evolve by the \emph{Ricci-Bourguignon} flow (RB flow for short) if $g(t)$ satisfies the following evolution equation

\begin{align}\label{RBfloweqn}
\frac{\pt g}{\pt t}=-2(\text{Ric}-\rho Rg),
\end{align}
where $\text{Ric}$ is the Ricci tensor of the metric, $R$ is the scalar curvature and $\rho \in \mathbb{R}$ is a constant. The flow in equation \eqref{RBfloweqn} was first introduced by Jean-Pierre Bourguignon \cite{bourguignon}, building on some unpublished work of Lichnerowicz and a paper of Aubin \cite{aubin}. We note that \eqref{RBfloweqn} is precisely the Ricci flow for $\rho=0$. In particular, the right hand side of the evolution equation \eqref{RBfloweqn} is of special interest for different values of $\rho$, for example
\begin{itemize}
\item $\rho=\frac 12$, the Einstein tensor $\text{Ric}-\frac{R}{2}g$.
\item $\rho=\frac 1n$, the traceless Ricci tensor $\tRic-\frac{R}{n}g$.
\item $\rho=\frac{1}{2(n-1)}$, the Schouten tensor $\tRic-\frac{R}{2(n-1)}g$.
\item $\rho =0$, the Ricci tensor $\tRic$.
\end{itemize} 

\medskip

A systematic study of the parabolic theory of the RB flow was initiated in \cite{catino1}. In that paper, the authors proved, along with many other results, the short time existence of the flow \eqref{RBfloweqn} on any closed $n$-dimensional manifold starting with an arbitrary initial metric $g_0$ for $\rho < \frac{1}{2(n-1)}$. Just like the Ricci flow case, we make the following  

\begin{defn} 
A Ricci-Bourguignon soliton (RB soliton for short) is a Riemannian manifold $(M^n,g)$ endowed with a vector field $X$ on $M$ that satisfies
\begin{align}\label{RBsolitoneqn}
R_{ij}+\frac 12(\mathcal{L}_Xg)_{ij}&=\lambda g_{ij}+\rho Rg_{ij},
\end{align}
where $\lieX g$ denotes the Lie derivative of the metric $g$ with respect to the vector field $X$ and $\lambda \in \mathbb{R}$ is a constant.
\end{defn}

 When $X=\nabla f$ for some smooth $f:M\rightarrow \mathbb{R}$, then $(M,g)$ is called a \emph{gradient} RB soliton. The soliton is called 
\begin{enumerate}
\item expanding when $\lambda<0$,
\item steady when $\lambda =0$,
\item shrinking when $\lambda >0$.
\end{enumerate}

\medskip

RB solitons correspond to self-similar solutions of the RB flow. An RB soliton is called \emph{trivial} if $X$ is a Killing vector field, i.e., $\lieX g=0$. We remark that even though the short time existence result for the flow \eqref{RBfloweqn} is for $\rho < \frac{1}{2(n-1)}$, any value of $\rho$ is possible for the considerations of self-similar solutions of the flow.

\medskip

Gradient RB solitons were studied in detail, for example in \cite{catino2} and \cite{catino3}, where the authors called them \emph{gradient} $\rho$-\emph{Einstein solitons}. Various classification and rigidity results about gradient RB solitons were proved in those papers and we refer the reader to those papers for precise statements and proofs of the results.

\medskip

The notion of Ricci \emph{almost} solitons was introduced in \cite{pigola}, where the authors modified the definition of a Ricci soliton by considering the parameter $\lambda$ in the definition of a Ricci soliton to be a \emph{function} rather than a constant. 
 Motivated by the Ricci flow case, we make the following
 
 \begin{defn}
 A Riemannian manifold $(M^n,g)$ is a Ricci-Bourguignon \emph{almost} soliton (RB almost soliton for short) if there is a vector field $X$ and a soliton function $\lambda :M\rightarrow \mathbb{R}$ satisfying
 \begin{align}\label{almostRBeqn}
 \tRic + \frac 12 \lieX g = \lambda g + \rho Rg.
\end{align}  
 \end{defn}
 An RB almost soliton is called a \emph{gradient} RB almost soliton if $X=\del f$ for some smooth function $f$ on $M$ and is expanding, steady or shrinking if $\lambda <0,\ \lambda =0$ or $\lambda >0$, respectively. We note that if $X$ is a Killing vector field, then a RB almost soliton is just a RB soliton as it forces $\lambda$ to be a constant. 
 
 Recall that a vector field $Y$ on a Riemannian manifold $(M,g)$ is called a conformal vector field if there exists a function $\psi:M\rightarrow \mathbb{R}$ such that 
 \begin{align*}
 \mathcal{L}_Y g & = 2\psi g   .
\end{align*}  
The function $\psi$ is sometimes referred as the potential of the vector field $Y$. The conformal vector field is nontrivial if $\psi \neq 0$. If $\psi=0$, then $Y$ is called a Killing vector field. 
 \medskip
 
Some characterization results for compact Ricci and Ricci almost solitons were obtained in \cite{barros2} and \cite{barros1} respectively. The goal of the present paper is to generalize the results obtained in those papers to RB and RB almost solitons. More precisely, in \cref{proofs} we prove the following theorems.

\medskip

\begin{thm}\label{mainthm1}
Let $(M^n,g,X,\lambda, \rho)$, $n\geq 3$, be a RB soliton and suppose that the vector field $X$ is a conformal vector field.
\begin{enumerate}
\item If $M$ is compact, then $X$ is a Killing vector field and hence $(M^n,g,X,\lambda, \rho)$ is a trivial RB soliton.
\item There is no nontrivial, complete noncompact RB soliton $(M^n, g, X, \lambda, \rho)$ with conformal vector field $X$.
\end{enumerate}
\end{thm}

\noindent
This generalizes Theorem 3 in \cite{barros2} and characterizes compact RB solitons when $X$ is a conformal vector field. The following corollary gives a lower bound for the first eigenvalue of the Laplacian on a compact RB soliton when $X$ is a  conformal vector field and generalizes Theorem 4 in \cite{barros2}.

\medskip

\begin{corr}\label{mainthm1corr}
Let $(M^n,g,X,\lambda, \rho)$ be a compact RB soliton with $X$ a conformal vector field. If $n\geq 3$ and $\lambda+\rho R >0$, then the first eigenvalue $\lambda_1$ of the Laplacian satisfies $\lambda_1 \geq (\lambda +\rho R) \frac{n}{n-1}$. Moreover, equality occurs if and only if $M^n$ is isometric to a Euclidean  sphere.
\end{corr}

\medskip
The next theorem characterizes compact RB almost solitons with $X$ a conformal vector field and generalizes Theorem 2 in \cite{barros1}.

\medskip

\begin{thm}\label{mainthm2}
Let $(M^n,g,X,\lambda, \rho)$, $n\geq 3$, be a compact RB almost soliton. If $X$ is a nontrivial conformal vector field, then $M^n$ is isometric to a Euclidean  sphere.
\end{thm}

\medskip

\begin{ex}
Consider the Euclidean sphere $(S^n, g_{\textup{round}}, X, \lambda)$, where $g_{\textup{round}}$ is the round metric on $S^n$, $X$ is the projection of a non-zero constant vector field $\bar{X}$ on $\mathbb{R}^{n+1}$, over $S^n$ and $\lambda = 1-\rho(n-1) \Div X$, with $\rho\in \mathbb{R}$ a constant. Then $X$ is a conformal vector field on $S^n$ and is not a Killing vector field. Since $\tRic = (n-1)g_{\textup{round}}$ for $S^n$, $(S^n, g_{\textup{round}}, X, \lambda)$ is a RB almost soliton.
\end{ex}

\begin{ex}
Let $M=I\times_h \Sigma$ be the $h$-warped product of the real interval $I\subset \mathbb{R}$ with $0\in I$, and the Riemannian manifold $(\Sigma, g_{\Sigma})$ with dim $\Sigma=n$. That is, $M$ is an $(n+1)$-dimensional product manifold $I\times \Sigma$ with the metric
\begin{align*}
g_{M}=dt^2+h^2g_{\Sigma},
\end{align*} 
where $t$ is a global parameter on $I$ and $h$ is a positive function on $I$. Suppose $\Sigma$ is an Einstein manifold with $\tRic_{\Sigma} = -(n-1)a$ with $a<0$. Following \cite[Example 2.5]{pigola}, define a function $f(x,t)=f(t)$ by
\begin{align}\label{ex1}
f(t)=B+\int_{0}^t h(s) \left[ A+(n-1)\int_{0}^s \frac{h''h-(h')^2-a}{h^3} dx   \right ]ds
\end{align}
for some constants $A, B\in \mathbb{R}$. Also define, $\lambda(x,t)=\lambda(t)$ by
\begin{align}\label{ex2}
\lambda(t) = -(n-1)\left[ \frac{(h')^2+a}{h^2} -an\rho   \right] -\frac{h''}{h}+h' \left[ A+(n-1)\int_{0}^s \frac{h''h-(h')^2-a}{h^3} dx   \right ],
\end{align}
where $\rho \in \mathbb{R}$ is a constant. It follows from \cite{pigola} that $\left(M, g_{M},  f'\frac{\partial}{\partial t}, \lambda, \rho \right)$ is a RB almost soliton, where $f$ is given by \eqref{ex1} and $\lambda$ is given by \eqref{ex2}.
\end{ex}

\medskip

The next theorem generalizes Theorem 3 in \cite{barros1} obtained for compact Ricci almost solitons, which is the case when $\rho=0$.

\medskip

\begin{thm}\label{mainthm3}
Let $(M^n,g,X,\lambda, \rho)$, $n\geq 3$, be a compact RB almost soliton. If $\rho \neq \frac {1}{n}$ and 
\begin{align}\label{mainthm3eqn}
\int_M [\tRic(X,X)+\frac{n\rho}{n\rho-1}\del_X \Div X-2\rho g(\del R, X) -\frac{(n(2\rho+1)-2)}{n\rho-1}g(\del \lambda, X)]dv \leq 0,
\end{align}
then $X$ is a Killing vector field and $M^n$ is a trivial RB soliton.
\end{thm}

\medskip

Since every RB almost soliton is also a RB soliton for constant $\lambda$, hence using $\del \lambda =0$, we get the following corollary for compact RB solitons.

\begin{corr}\label{mainthm4}
Let $(M^n,g,X,\lambda, \rho)$, $n\geq 3$, be a compact RB soliton. If $\rho \neq \frac 1n$ and
\begin{align}\label{mainthm4eqn}
\int_M [\tRic(X,X)+\frac{n\rho}{(n\rho-1)}\del_X \Div X-2\rho g(\del R, X)]dv\leq 0,
\end{align}
then $X$ is a Killing vector field and $M^n$ is a trivial RB soliton.
\end{corr}

\medskip

\begin{rem}
Corollary~\ref{mainthm4} is an analog of Theorem 1.1 in \cite{petersen-wylie}, which was for the case of compact Ricci solitons. We obtain Petersen-Wylie's result from our result by taking $\rho=0$. In fact, the condition in \eqref{mainthm4eqn} is analogous to the condition in \cite[Theorem 1.1]{petersen-wylie}, which is obtained when $\rho=0$ in \eqref{mainthm4eqn}.
\end{rem}

\medskip

Finally, we obtain integral formulas for compact gradient RB almost solitons, generalizing a corresponding result for compact gradient Ricci almost solitons from  \cite{barros1}.  
 
 \medskip

\begin{thm}\label{mainthm6}
Let $(M^n,g,\del f, \lambda, \rho)$ be a compact gradient RB almost soliton. Then

\begin{align}
\int_M  \left|\del^2f-\frac{\Delta f}{n}g\right|^2 dv &=\frac{(n-2)}{2n}\int_M g(\del R, \del f) dv  
\end{align}
and                   
\begin{align}
\int_M \left|\tRic -\frac Rn g\right|^2 dv & = \frac{(n-2)}{2n}\int_M g(\del R, \del f) dv  . 
\end{align}
\end{thm}

\medskip

As an application of the previous theorem, we provide some conditions for a compact gradient RB almost soliton to be isometric to a Euclidean  sphere.

\begin{corr}\label{mainthm6corr}
A nontrivial compact gradient RB almost soliton $(M^n, g, \del f, \lambda, \rho)$, $n\geq 3$, is isometric to a Euclidean  sphere if any of the following assertions hold:
\begin{enumerate}
\item $M^n$ has constant scalar curvature.
\item $\int_M g(\del R, \del f) dv\leq 0$.
\item $M^n$ is a homogenous manifold.
\end{enumerate}
\end{corr}

\medskip 

The paper is organized as follows. In \cref {prelims}, we state and prove some identities for RB solitons and RB almost solitons which will be used to prove the main results. 
In \cref{proofs}, we prove the main theorems and their corollaries.

\begin{ack}

The author would like to thank his advisor Spiro Karigiannis for his constant encouragement and advice. The author is also grateful to the anonymous referee for various useful remarks and suggestions which have improved the quality of the paper.
\end{ack}

\section{Preliminaries}\label{prelims}

In this section, we prove some general results about RB solitons and almost solitons. The proofs of some of these results in the compact gradient case can also be found in \cite{catino2} or \cite{catino3}. Let us first recall the Ricci identity for a $(0,2)$-tensor $\alpha$:
\begin{align*}
\del_i\del_j \alpha_{kl}-\del_j\del_i \alpha_{kl}&= -R_{ijkm}\alpha_{ml}-R_{ijlm}\alpha_{km},
\end{align*}
where $R_{ijkl}$ is the Riemann curvature tensor. The Ricci curvature is obtained from the Riemann curvature tensor by contracting on the first and last index
\begin{align*}
R_{ij}&=g^{kl}R_{kijl}
\end{align*} 
and the contracted second Bianchi identity is 
\begin{align*}
\del_iR_{ij}&=\frac 12 \del_jR.
\end{align*}

We start with the following 

\medskip
 
\begin{prop} \label{prop2.1}
Let $(M^n,g,\del f,\lambda, \rho)$ be a gradient RB almost soliton. Then the following identities hold

\begin{flalign}
(1-n\rho)R+\Delta f &=n\lambda. &&  \label{prop2.1.1} \\
(1 -2\rho (n-1))\del_iR &= 2R_{il}\del_lf+2(n-1)\del_i\lambda. && \label{prop2.1.2} \\
\del_jR_{ik}-\del_kR_{ij} &= R_{jkil}\del_lf+\rho (\del_jRg_{ik}-\del_kRg_{ij}) \nonumber \\
& \quad +(\del_j\lambda g_{ik}-\del_k \lambda g_{ij}). && \label{prop2.1.3} \\
\del_i \big[(1 -2\rho(n-1))R+|\del f|^2-2(n-1)\lambda \big]&= (2\rho R+2\lambda)\del_if. && \label{prop2.1.4}
\end{flalign}

\end{prop}

\begin{proof}
For a gradient RB almost soliton we have
\begin{align}\label{gradRBsoliton}
R_{ij}+\del_i\del_jf&=\lambda g_{ij}+\rho Rg_{ij} .
\end{align}
Taking the trace of the above equation gives \eqref{prop2.1.1}. Next, taking the covariant derivative of \eqref{prop2.1.1} with respect to an orthonormal frame gives
\begin{align*}
(1-n\rho)\del_iR+\del_i\del_j\del_jf=n\del_i\lambda .
\end{align*}
Commuting the covariant derivatives and using the contracted second Bianchi identity, we obtain
\begin{align}
(1-n\rho)\del_iR &= -\del_j\del_i\del_jf+R_{il}\del_lf+n\del_i\lambda \nonumber \\
&= -\del_j(-R_{ij}+\lambda g_{ij}+\rho Rg_{ij})+R_{il}\del_lf +n\del_i\lambda\nonumber \\
&= \frac 12 \del_iR -\rho \del_iR-\del_i\lambda+R_{il}\del_l f+n\del_i\lambda \nonumber 
\end{align}
and hence
\begin{align}
(\frac 12 -\rho (n-1))\del_iR = R_{il}\del_lf+(n-1)\del_i\lambda,
\end{align}
which proves \eqref{prop2.1.2}.

\medskip

For proving \eqref{prop2.1.3}, we use \eqref{gradRBsoliton} and commute the covariant derivatives to get 
\begin{align}
\del_jR_{ik}-\del_kR_{ij}&=(\del_k\del_i\del_jf-\del_j\del_i\del_kf)+\rho (\del_jRg_{ik}-\del_kRg_{ij}) \nonumber \\
& \quad +(\del_j\lambda g_{ik}-\del_k \lambda g_{ij})\nonumber \\
&=(\del_k\del_j\del_if-\del_j\del_k\del_if)+\rho (\del_jRg_{ik}-\del_kRg_{ij}) \nonumber \\
& \quad +(\del_j\lambda g_{ik}-\del_k \lambda g_{ij}) \nonumber \\
&=R_{jkil}\del_lf+\rho (\del_jRg_{ik}-\del_kRg_{ij})+(\del_j\lambda g_{ik}-\del_k \lambda g_{ij}).
\end{align}

\medskip

Finally, from \eqref{prop2.1.2} we get
\begin{align}
(1 -2\rho(n-1))\del_iR&=2\del_lf(-\del_i\del_lf+\lambda g_{il}+\rho Rg_{il})+2(n-1)\del_i \lambda \nonumber \\
&=-2\del_lf\del_i\del_lf+2\lambda \del_if+2\rho R\del_if +2(n-1)\del_i \lambda\nonumber \\
&=-\del_i|\del_lf|^2+2\lambda \del_if+2\rho R\del_if+2(n-1)\del_i \lambda, \nonumber
\end{align}
so we get
\begin{align}
\del_i \big[(1 -2\rho(n-1))R+|\del f|^2-2(n-1)\lambda \big ]&= (2\rho R+2\lambda)\del_if,
\end{align}
which proves \eqref{prop2.1.4}.

\end{proof}

\begin{rem}\label{prop2.5}
The analogous identities for gradient RB solitons $(M^n,g, \del f, \lambda, \rho)$ are
\begin{flalign}
(1-n\rho)R+\Delta f& =n\lambda. && \label{2.5.1} \\
(1-2\rho (n-1))\del_iR &=2R_{il}\del_lf. && \label{2.5.2} \\
\del_jR_{ik}-\del_kR_{ij} & =R_{jkil}\del_lf+\rho (\del_jRg_{ik}-\del_kRg_{ij}). && \label{2.5.3} \\
\del_i\big [(1-2\rho(n-1))R+|\del f|^2-2\lambda f    \big ] & =2\rho R\del_if. && \label{2.5.4}
\end{flalign}
The proofs of these identities are special cases of the previous result as $\del \lambda =0$.

\end{rem}

\medskip

We recall the following lemma from \cite[Lemma 2.1]{petersen-wylie}.

\begin{lemma}
Let $X$ be a vector field on a Riemannian manifold $(M^n,g)$. Then
\begin{align}\label{lemma2.2}
\Div(\lieX g)(X) &= \frac 12 \Delta |X|^2-|\del X|^2 + \text{Ric} (X,X)+\del_X \Div X .
\end{align}
When $X=\del f$ and $Z$ is any vector field, then
\begin{align}
\Div (\mathcal{L}_{\del f}g)(Z)=2\text{Ric}(Z,\del f)+2\del_Z \Div \del f .
\end{align}
\end{lemma}

\vspace{0.3cm}

We use the preceding lemma to prove the following

\medskip

\begin{lemma}\label{lemma2.3}
Let $(M^n,g,X,\lambda, \rho)$ be a RB almost soliton. Then

\begin{align}\label{2.3.1}
\frac{(1-n\rho)}{2}\Delta |X|^2 &=(1-n\rho)|\del X|^2 +(n\rho -1)\tRic (X,X)+n\rho \del_X \Div X \nonumber \\
& \quad + 2\rho(1-n\rho)g(\del R,X)-(n(2\rho+1)-2)g(\del \lambda, X)
\end{align}
and
\begin{align}\label{2.3.2}
\frac{(1-n\rho)}{2}(\Delta-\del_X)|X|^2 & = (1-n\rho)|\del X|^2+\lambda (n\rho-1)|X|^2+\rho(n\rho-1)R|X|^2 \nonumber \\
& \quad +n\rho \del_X \Div X + 2\rho(1-n\rho)g(\del R,X) \nonumber \\
& \quad -(n(2\rho+1)-2)g(\del \lambda, X).
\end{align}

\end{lemma}

\medskip

\begin{proof}
We first notice that \eqref{almostRBeqn} gives
\begin{align}\label{divric}
2\Div \tRic + \Div (\lieX g) = 2\del \lambda+2\rho \del R.
\end{align}
Taking the trace of \eqref{almostRBeqn} gives $(1-n\rho)R + \Div X = n\lambda$ and thus
\begin{align}\label{delR}
(1-n\rho)\del_XR+\del_X(\Div X)=n\del_X \lambda.
\end{align}

\medskip

So, using \eqref{lemma2.2}, \eqref{divric}, \eqref{delR} and the contracted second Bianchi identity, we get

\begin{align}
\del_X(\Div X)&=(n\rho-1)\del_XR +ng(\del \lambda, X) \nonumber \\
&=2(n\rho-1)\Div \tRic(X) +ng(\del \lambda, X)\nonumber \\
&=-(n\rho-1)\Div (\lieX g)(X)+2\rho(n\rho -1)g(\del R, X) +2(n\rho-1)g(\del \lambda, X)\nonumber \\
& \quad +ng(\del \lambda, X) \nonumber \\
&=(1-n\rho )\Big ( \frac 12 \Delta |X|^2-|\del X|^2 + \tRic (X,X)+\del_X \Div X   \Big ) \nonumber \\
& \quad +2\rho (n\rho -1)g(\del R,X)+(n(2\rho+1)-2)g(\del \lambda, X) \nonumber \\
&= \frac{(1-n\rho)}{2} \Delta |X|^2-(1-n\rho)|\del X|^2+(1-n\rho)\tRic (X,X) \nonumber \\
& \quad +(1-n\rho) \del_X \Div X + 2\rho(n\rho-1)g(\del R,X)+(n(2\rho+1)-2)g(\del \lambda, X), \nonumber
\end{align}
which gives

\medskip

\begin{align}
\frac{(1-n\rho)}{2}\Delta |X|^2 &=(1-n\rho)|\del X|^2 +(n\rho -1)\tRic (X,X)+n\rho \del_X \Div X \nonumber \\
& \quad + 2\rho(1-n\rho)g(\del R,X)-(n(2\rho+1)-2)g(\del \lambda, X),
\end{align}
thus proving \eqref{2.3.1}.

\medskip

Using \eqref{almostRBeqn} to write $\tRic(X,X)=-\frac 12 (\lieX g)(X,X)+\lambda |X|^2+\rho R|X|^2$ in \eqref{2.3.1}, we get

\begin{align}
\frac{(1-n\rho)}{2}\Delta |X|^2 &=(1-n\rho)|\del X|^2 +(n\rho -1)\Big(-\frac 12 (\lieX g)(X,X)+\lambda |X|^2+\rho R|X|^2 \Big) \nonumber \\
& \quad +n\rho \del_X \Div X + 2\rho(1-n\rho)g(\del R,X)-(n(2\rho+1)-2)g(\del \lambda, X) \nonumber \\
&= (1-n\rho)|\del X|^2 +\frac{(1-n\rho)}{2}\del_X|X|^2+\lambda (n\rho-1)|X|^2+\rho(n\rho-1)R|X|^2 \nonumber \\
& \quad +n\rho \del_X \Div X + 2\rho(1-n\rho)g(\del R,X) -(n(2\rho+1)-2)g(\del \lambda, X), \nonumber 
\end{align}
which gives
\begin{align}
\frac{(1-n\rho)}{2}(\Delta-\del_X)|X|^2 & = (1-n\rho)|\del X|^2+\lambda (n\rho-1)|X|^2+\rho(n\rho-1)R|X|^2 \nonumber \\
& \quad +n\rho \del_X \Div X + 2\rho(1-n\rho)g(\del R,X) \nonumber \\
& \quad -(n(2\rho+1)-2)g(\del \lambda, X),
\end{align}
proving \eqref{2.3.2}.
\end{proof}

\medskip

If we consider the diffusion operator $\Delta_X=\Delta-\del_X$, then from the previous lemma with $X=\del f$ and $\Delta_f=\Delta-\del_{\del f}$, we obtain the following corollary.

\begin{corr}
For a gradient RB almost soliton $(M^n, g, \del f, \lambda, \rho)$, we have
\begin{align}
\frac{(1-n\rho)}{2}\Delta_f|\del f|^2 &= (1-n\rho)|\del^2f|^2+\lambda(n\rho-1)|\del f|^2+\rho(n\rho-1)R|\del f|^2 \nonumber \\
& \quad +n\rho \del_{\del f}(\Delta f)+2\rho (1-n\rho)g(\del R, \del f) \nonumber \\
& \quad -(n(2\rho +1)-2)g(\del \lambda, \del f) .
\end{align}
\end{corr}

\medskip

\begin{rem}\label{lemma2.6}
The analogs of \eqref{2.3.1} and \eqref{2.3.2} for a RB soliton $(M^n,g,X,\lambda, \rho)$ are
\begin{align}\label{2.6.1}
\frac{(1-n\rho)}{2}\Delta |X|^2 &=(1-n\rho)|\del X|^2 +(n\rho -1)\tRic (X,X)+n\rho \del_X \Div X \nonumber \\
& \quad + 2\rho(1-n\rho)g(\del R,X)
\end{align}
and
\begin{align}\label{2.6.2}
\frac{(1-n\rho)}{2}(\Delta-\del_X)|X|^2 & = (1-n\rho)|\del X|^2+\lambda (n\rho-1)|X|^2+\rho(n\rho-1)R|X|^2 \nonumber \\
& \quad +n\rho \del_X \Div X + 2\rho(1-n\rho)g(\del R,X) . \nonumber \\
\end{align}
The proofs are special cases of the proof of Lemma \ref{lemma2.3} with $\del \lambda =0$.
\end{rem}

\medskip

\section{Proofs of the Results}\label{proofs}

We start this section by proving the following lemma which will be used in the proofs of Theorem~\ref{mainthm1} and Theorem~\ref{mainthm2}.

\begin{lemma}\label{lemma3.1}
Let $(M^n,g,X,\lambda, \rho)$, $n\geq 3$, be a RB almost soliton. If $X$ is a nontrivial conformal vector field with $\lieX g=2\psi g$, then $R$ and $\lambda-\psi$ are constant.
\end{lemma}

\begin{proof}
The soliton equation is 
\begin{align}\label{almostconf1}
R_{ij}+\frac 12(\lieX g)_{ij}=\lambda g_{ij}+\rho Rg_{ij},
\end{align}
where $\lambda:M\rightarrow \mathbb{R}$ is a function. If $X$ is a nontrivial conformal vector field, then we have
\begin{equation}
\lieX g=2\psi g,
\end{equation}
for some function $\psi:M\rightarrow \mathbb{R}$, $\psi\neq 0$. So \eqref{almostconf1} becomes
\begin{align}\label{almostconf2}
R_{ij}=(\lambda-\psi+\rho R)g_{ij}.
\end{align}
Taking the divergence of \eqref{almostconf2}, we get
\begin{align*}
\del_iR_{ij}&=\del_i(\lambda-\psi+\rho R)g_{ij},
\end{align*}
which implies
\begin{align}
(\frac 12 -\rho)\del_jR &= \del_j (\lambda-\psi). \label{almostconf3}
\end{align}

\medskip

On the other hand, tracing \eqref{almostconf2} and taking the covariant derivative, we get
\begin{align}\label{almostconf4}
(1-n\rho)\del_jR &=n\del_j(\lambda-\rho).
\end{align}

\medskip

So from \eqref{almostconf3} and \eqref{almostconf4}, we get
\begin{align}
(1-n\rho)\del_jR&=n(\frac 12-\rho)\del_jR .
\end{align}
Thus, if $M$ is connected, then $R$ is constant and hence $\lambda-\psi$ is constant.
\end{proof}

\begin{rem}\label{rem3.2}
If $(M^n,g,X,\lambda, \rho)$, $n\geq 3$, is a RB soliton and $X$ is a conformal vector field with $\lieX g=2\psi g$ for some function $\psi:M\rightarrow \mathbb{R}$,  then the proof of Lemma~\ref{lemma3.1} shows that $R$ and $\psi$ are constant as in this case $\del \lambda=0$.
\end{rem}

We prove Theorem~\ref{mainthm1} which we restate here.
\begin{thm}
Let $(M^n,g,X,\lambda, \rho)$, $n\geq 3$, be a RB soliton and suppose that the vector field $X$ is a conformal vector field.
\begin{enumerate}
\item If $M$ is compact, then $X$ is a Killing vector field and hence $(M^n,g,X,\lambda, \rho)$ is a trivial RB soliton.
\item There is no nontrivial, complete noncompact RB soliton $(M^n, g, X, \lambda, \rho)$ with conformal vector field $X$.
\end{enumerate}
\end{thm}

\begin{proof}
Suppose $X$ is a conformal vector field with potential $\psi:M \rightarrow \mathbb{R}$, i.e., 
\begin{align}\label{Xpsi}
\lieX g&=2\psi g,
\end{align}
then from Remark~\ref{rem3.2} we know that $R$ and $\psi$ are constant.

Taking the trace of \eqref{Xpsi}, we get
\begin{equation*}
2\Div X = 2n\psi, 
\end{equation*}
which upon integration over compact $M$ gives
\begin{equation}
0=\int_M 2\Div X dv=2n\Vol(M)\psi,
\end{equation}
which implies that $\psi =0$. So $X$ is a Killing vector field and hence $(M^n,X,g,\lambda, \rho)$ is a trivial RB soliton.

\medskip

If $M$ is noncompact and a gradient RB soliton with $X=\del f$, then $X$ being conformal implies
\begin{align*}
\del_i\del_jf=\psi g_{ij}
\end{align*}
and by Remark~\ref{rem3.2}, $\psi$ is constant. If $\psi =0$, then $X$ is a Killing vector field and $M$ is a trivial RB soliton. If $\psi\neq 0$, then from \cite[Theorem 2]{tashiro}, we conclude that $M^n$ is isometric to the Euclidean space. 

\end{proof}

Next we prove Corollary \ref{mainthm1corr}.

\begin{proof}
Since $M^n$ is compact, we know from Theorem \ref{mainthm1} that $X$ is a Killing vector field and hence we have $\tRic=(\lambda+\rho R)g$. So we can apply a classical theorem due to Lichnerowicz \cite{Lich}, which states that if $\tRic\geq k$, where $k>0$ is a constant, then the first eigenvalue of the Laplacian $\lambda_1$ satisfies $\lambda_1\geq \frac{n}{n-1}k$. So we get
\begin{align*}
\lambda_1\geq (\lambda+\rho R)\frac{n}{n-1}   .            
\end{align*}

Moreover, for the equality case, we can apply Obata's theorem \cite{obata}, to conclude that equality occurs in the above inequality if and only if $M^n$ is isometric to a Euclidean sphere of constant curvature $\frac{(\lambda+\rho R)}{n-1}$.
\end{proof}

\medskip

We now prove Theorem~\ref{mainthm2}, which we restate here. 

\begin{thm}
Let $(M^n,g,X,\lambda, \rho)$, $n\geq 3$, be a compact RB almost soliton. If $X$ is a nontrivial conformal vector field, then $M^n$ is isometric to a Euclidean  sphere.
\end{thm}

\begin{proof}
Suppose $X$ is a nontrivial conformal vector field with potential function $\psi:M\rightarrow \mathbb{R}$, i.e.,
\begin{align*}
\lieX g=2\psi g,
\end{align*}
with $\psi\neq 0$. Since $(M^n,g,X,\lambda, \rho)$ is a compact RB almost soliton with $n\geq 3$, Lemma~\ref{lemma3.1} tells us that $R$ and $\lambda-\psi$ are constant. So from Lemma 2.3 in \cite[pg.52]{yano}, we conclude that $R\neq 0$ or else $\psi$ would be 0. Taking the Lie derivative of \eqref{almostconf2}, we get
\begin{align*}
\lieX \tRic &= \lieX (\lambda-\psi+\rho R)g 
\end{align*}
and since $(\lambda-\psi),\ \rho$ and $R$ are all constant, so we get
\begin{align}\label{almostconf5}
\lieX \tRic &=2(\lambda-\psi+\rho R)\psi g  . 
\end{align}

Now we can apply Theorem 4.2 of \cite[pg. 54]{yano}, to conclude that $M$ is isometric to a Euclidean  sphere.
\end{proof}

\medskip

We proceed to the proof of Theorem \ref{mainthm3}.

\begin{proof}
We see from \eqref{2.3.1} of Lemma \ref{lemma2.3} that 
\begin{align*}
\frac{(1-n\rho)}{2}\Delta |X|^2 &=(1-n\rho)|\del X|^2 +(n\rho -1)\tRic (X,X)+n\rho \del_X \Div X \nonumber \\
& \quad + 2\rho(1-n\rho)g(\del R,X)-(n(2\rho+1)-2)g(\del \lambda, X).
\end{align*}

Integrating above over compact $M$, we get
\begin{align}
0&=\int_M [(1-n\rho)|\del X|^2+ (n\rho-1)\tRic (X,X)+n\rho \del_X\Div X \nonumber \\
& \qquad + 2\rho(1-n\rho)g(\del R,X)-(n(2\rho+1)-2)g(\del \lambda, X)] dv   .
\end{align}

Since $\rho \neq \frac 1n$, we get
\begin{align}
\int_M |\del X|^2 dv &= \int_M [\tRic(X,X)+\frac{n\rho}{n\rho-1}\del_X \Div X-2\rho g(\del R, X) \nonumber \\
& \qquad -\frac{(n(2\rho+1)-2)}{n\rho-1}g(\del \lambda, X)]dv,
\end{align}
so if \eqref{mainthm3eqn} holds, then $|\del X|^2=0$ and hence $X$ is a Killing vector field. Thus, $(M^n, g, X, \lambda, \rho)$ is trivial.
\end{proof}

\medskip

The proof of Corollary~\ref{mainthm4} is a special case of the proof of Theorem~\ref{mainthm3}, where we use \eqref{2.6.1} of Remark~\ref{lemma2.6}. 

\medskip

Next, we prove Theorem \ref{mainthm6} which we restate here.

\begin{thm}
Let $(M^n,g,\del f, \lambda, \rho)$ be a compact gradient RB almost soliton. Then

\begin{align}\label{mainthm6.1}
\int_M  \left|\del^2f-\frac{\Delta f}{n}g\right|^2 dv &=\frac{(n-2)}{2n}\int_M g(\del R, \del f) dv
\end{align}
and                    
\begin{align}\label{mainthm6.2}
\int_M \left|\tRic -\frac Rn g\right |^2 dv & = \frac{(n-2)}{2n}\int_M g(\del R, \del f) dv.
\end{align}
\end{thm}

\begin{proof}

For proving \eqref{mainthm6.1}, we take the divergence of \eqref{prop2.1.4} of Proposition \ref{prop2.1} to get 
\begin{align}\label{thm6.1}
(1-2\rho(n-1))\Delta R + \Delta |\del f|^2-2(n-1)\Delta \lambda &=2\rho g(\del R, \del f) + 2g(\del \lambda, \del f) \nonumber \\
& \quad + (2\rho R+2\lambda)\Delta f     .
\end{align}

By commuting the covariant derivatives, we have
\begin{align*}
\del_i\del_i(g(\del_jf,\del_jf))&=2\del_i(g(\del_i\del_jf, \del_jf)) \\
&=2g(\del_i\del_i\del_jf, \del_jf)+2|\del^2f|^2\\
&=2g(\del_j\del_i\del_if-R_{ijil}\del_lf, \del_jf)+2|\del^2f|^2\\
&= 2g(\del(\Delta f), \del f)+2\tRic(\del f, \del f)+2|\del^2f|^2,
\end{align*}
so \eqref{thm6.1} becomes

\begin{equation*}
 (1-2\rho(n-1))\Delta R+2g(\del(\Delta f), \del f)+2\tRic(\del f, \del f)+2|\del^2f|^2-2(n-1)\Delta \lambda=
\end{equation*}
\begin{align}\label{mainthm5.2}
\qquad \qquad 2\rho g(\del R, \del f)+ 2g(\del \lambda, \del f) + (2\rho R+2\lambda)\Delta f  .
\end{align}

From \eqref{prop2.1.1} of Proposition~\ref{prop2.1.1}, we know that $\Delta f=n\lambda + (n\rho-1)R$, which on differentiation and using \eqref{gradRBsoliton} becomes

\begin{align*}
0&=\del_i \Delta f+(1-n\rho)\del_i R-n\del_i \lambda & \\
&=(1-n\rho)\del_iR+\del_j\del_i\del_jf-R_{il}\del_lf-n \del_i \lambda \\
&=(1-n\rho)\del_iR+\del_j(-R_{ij}+\lambda g_{ij}+\rho Rg_{ij})-R_{il}\del_lf-n\del_i \lambda\\
&=(\frac 12-\rho(n-1))\del_iR-R_{il}\del_lf+(1-n)\del_i \lambda
\end{align*}
and hence

\begin{align}\label{thm6.2}
2\tRic(\del f, \del f)&=(1-2\rho(n-1))g(\del R, \del f)+2(1-n)g(\del \lambda, \del f)  .
\end{align}

So, using \eqref{thm6.2} and $\Delta f=n\lambda+(n\rho-1)R$, the left hand side of \eqref{mainthm5.2} becomes

\begin{align*}
(1-2\rho(n-1))\Delta R+2|\del^2f|^2-2(n-1)\Delta \lambda+2g(\del \lambda, \del f)+(2\rho -1)g(\del R, \del f)
\end{align*}
and hence \eqref{mainthm5.2} becomes

\begin{align}\label{mainthm6.2.}
(1-2\rho(n-1))\Delta R+2|\del^2f|^2-2(n-1)\Delta \lambda &= g(\del R, \del f)+(2\rho R+2\lambda) \Delta f  .
\end{align}

Since $|\del^2f-\frac{\Delta f}{n}g|^2 = |\del^2f|^2-\frac{(\Delta f)^2}{n}$, \eqref{mainthm6.2.} becomes

\begin{align}\label{thm6.3}
(1-2\rho(n-1))\Delta R+2\left|\del^2f-\frac{\Delta f}{n}g \right|^2&=g(\del R, \del f)+(2\rho R+2\lambda) \Delta f-2\frac{(\Delta f)^2}{n} \nonumber \\
& \quad +2(n-1)\Delta \lambda \nonumber \\
&=g(\del R, \del f)+(2\rho R+2\lambda) \Delta f \nonumber \\
& \quad -2\frac{(\Delta f)}{n}(n\lambda + (n\rho-1)R)+2(n-1)\Delta \lambda \nonumber \\
&= g(\del R, \del f)+\frac 2n R\Delta f+2(n-1)\Delta \lambda  .
\end{align}
Integrating \eqref{thm6.3} over compact $M$, we obtain
\begin{align}\label{mainthm5.4}
\int_M  2\left|\del^2f-\frac{\Delta f}{n}g\right |^2 dv &=\int_M \Big[ g(\del R, \del f)+\frac 2n R\Delta f \Big ] dv \nonumber \\
&=\frac{(n-2)}{n}\int_M g(\del R, \del f) dv,
\end{align}
where we have used integration by parts in the first equality to get the second equality. This proves \eqref{mainthm6.1}. 

\medskip

For proving \eqref{mainthm6.2}, note that 
\begin{align}\label{mainthm5.5}
\tRic -\frac Rn g&=-\del^2f +\lambda g+\rho Rg-\frac Rn g \nonumber \\
&=-\del^2f+(\lambda+\rho R-\frac Rn)g \nonumber \\
&=-\del^2f+\frac{\Delta f}{n} g
\end{align}
and then \eqref{mainthm6.2} follows from \eqref{mainthm6.1}.
\end{proof}

\medskip

\begin{rem}
Since a gradient RB soliton is a special case of a gradient RB almost soliton, the proof of Theorem~\ref{mainthm6}, with $\del \lambda=0$, shows that the same integral formulas \eqref{mainthm6.1} and \eqref{mainthm6.2} hold for a compact gradient RB solitons as well.
\end{rem}

\medskip

Finally, using Theorem~\ref{mainthm6}, we prove  Corollary~\ref{mainthm6corr}.

\begin{proof}
Observe that any of the assertions of  Corollary~\ref{mainthm6corr} enable us to conclude that the right hand side of \eqref{mainthm6.2} is less than or equal to zero and hence $\tRic=\frac Rn g$. So, from \eqref{gradRBsoliton}, we see that 
\begin{align*}
\del_i\del_jf&=(\lambda +R(\rho-\frac 1n))g
\end{align*}
and hence $\del f$ is a nontrivial conformal vector field, so from Theorem \ref{mainthm2}, we get that $M^n$ is isometric to a Euclidean  sphere.
\end{proof}

\vspace{0.5cm}

\bibliographystyle{amsalpha}
\bibliography{RB}
\end{document}